\documentclass[letterpaper, 10 pt, conference]{ieeeconf}  % Comment this line out
\IEEEoverridecommandlockouts                              % This command is only
\usepackage{graphics} % for pdf, bitmapped graphics files
\usepackage{epsfig} % for postscript graphics files
\usepackage{mathptmx} % assumes new font selection scheme installed
\usepackage{times} % assumes new font selection scheme installed
\usepackage{amsmath} % assumes amsmath package installed
\usepackage{amssymb}  % assumes amsmath package installed
\usepackage{xcolor}

\newcommand{\at}{\textup{ArcTanh }} 

\newtheorem{proposition}{Proposition}
\newtheorem{theorem}{Theorem}
\newtheorem{lemma}{Lemma}
\newtheorem{remark}{Remark}

\title{\LARGE \bf
Lockdown interventions in SIR models:\\ Is the reproduction number the right control variable?
}

\author{Leonardo Cianfanelli$^{1,*}$, Francesca Parise$^{2,*}$, Daron Acemoglu$^{3}$, Giacomo Como$^{1,4}$, Asuman Ozdaglar$^5$
\thanks{$^1$Dipartimento di Scienze Matematiche, Politecnico di Torino}%
\thanks{$^2$Department of Electrical and Computer Engineering, Cornell University}
\thanks{$^3$Department of Economics, Massachusetts Institute of Technology, MIT}
\thanks{$^4$Department of Automatic Control, Lund University}
\thanks{$^5$Department of Electrical Engineering and Computer Science, MIT}
\thanks{$^*$First two authors contributed equally, remaining ones are listed in alphabetical order.}%
\thanks{\it{leonardo.cianfanelli@polito.it, fp264@cornell.edu, daron@mit.edu, giacomo.como@polito.it, asuman@mit.edu}}}

\begin{document}

\maketitle
\thispagestyle{empty}
\pagestyle{empty}

%%%%%%%%%%%%%%%%%%%%%%%%%%%%%%%%%%%%%%%%%%%%%%%%%%%%%%%%%%%%%%%%%%%%%%%%%%%%%%%%
\begin{abstract}
The recent COVID-19 pandemic highlighted the need of non-pharmaceutical interventions in the first stages of a pandemic. Among these, lockdown policies proved unavoidable yet extremely costly from an economic perspective. To better understand the tradeoffs between economic and epidemic costs of lockdown interventions, we here focus on a simple SIR epidemic model and study lockdowns as solutions to an optimal control problem. We first show numerically that the optimal lockdown policy exhibits a phase transition from \emph{suppression} to \emph{mitigation} as the time horizon grows, i.e., if the horizon is short the optimal strategy is to impose severe lockdown to avoid diffusion of the infection, whereas if the horizon is long the optimal control steers the system to herd immunity to reduce economic loss. We then consider two alternative policies, motivated by government responses to the COVID-19 pandemic, where lockdown levels are selected to either stabilize the reproduction number (i.e., ``flatten the curve'') or the fraction of infected (i.e., containing the number of hospitalizations). We compute analytically the performance of these two feedback policies and compare them to the optimal control. Interestingly, we show that in the limit of infinite horizon stabilizing the number of infected is preferable to controlling the reproduction number, and in fact yields close to optimal performance.

\end{abstract}

%%%%%%%%%%%%%%%%%%%%%%%%%%%%%%%%%%%%%%%%%%%%%%%%%%%%%%%%%%%%%%%%%%%%%%%%%%%%%%%%
\section{INTRODUCTION}
The COVID-19 pandemic has revealed the need for non-pharmaceutical interventions such as lockdowns to mitigate the spread of novel diseases, when cures and vaccines are not available. Unfortunately, while very effective, such interventions are very costly from an economic perspective. A wide variety of strategies were adopted all over the world, implementing different tradeoffs between economic and epidemic cost. In general, this class of epidemic control strategies can be classified in two main categories: \emph{suppression strategies}, which “aim to reverse epidemic growth, reducing case numbers to low levels and maintaining that situation indefinitely", or
\emph{mitigation strategies}, which “focus on slowing but not necessarily stopping epidemic spread, reducing peak healthcare demand" while %protecting those most at risk of severe disease from infection", 
minimizing the negative effects on economic activities \cite{ferguson2020report}.

We here aim at understanding better the trade-off in terms of economic and epidemic cost  of lockdown strategies on epidemic spread within a standard susceptible-infected-recovered (SIR) model \cite{kermack1927contribution}. We study the optimal lockdown from an optimal control perspective as first analyzed in \cite{behncke2000optimal}, with an objective that takes into account the epidemic cost associated with deaths and the economic cost due to lockdowns, as suggested in \cite{alvarez2020simple,acemoglu2020multi}. The cost takes into account hospital congestion, which appeared to be a key feature of the recent pandemic, by assuming that the lethality due to the disease is an increasing affine function of the fraction of infected $i$,  leading to a quadratic mortality rate in $i$. To the best of our knowledge, no analytical solutions to optimal control problems with quadratic cost in $i$ are provided in the literature \cite{nowzari2016analysis}. While in \cite{alvarez2020simple,acemoglu2020multi} numerical solutions are provided, in \cite{kruse2020optimal} authors show that if the cost is linear in the infected, i.e., there is no hospital congestion, and convex in the intervention, the optimal policy is quasi-convex. In  \cite{miclo2020optimal} the authors consider the problem of minimizing the cost of the intervention subject to the constraint that the fraction of infected is below a certain threshold, i.e., $i(t) \le \overline{i}$ for every time $t$,  where the value of $\overline{i}$ is based on hospitals capacity. For the case of infinite time horizon $T$, the authors prove analytically that in a first phase it is optimal to let the epidemic spread, until $i(t)=\overline{i}$; then the optimal lockdown is such that the infected remain constant at the threshold, and is released when herd immunity is reached, i.e., when the fraction of people that contracted the disease is large enough to guarantee that, from there on, the number of infected people will decrease even without lockdown. While \cite{miclo2020optimal} handle hospital congestion by a hard constraint on the simultaneous fraction of infected, without any cost on the number of deaths, our formulation considers a quadratic cost in $i$ without hard constraints. A feedback control with a hard constraint of the simultaneous fraction of infected is also proposed in \cite{di2020covid}. Optimal control approaches to containment of COVID-19 pandemic using more detailed compartmental models are also studied numerically in \cite{djidjou2020optimal,kantner2020beyond}.

The contribution of our work is two-fold. We first show numerically that as $T$ increases and the trade-off between economic and epidemic cost vary, the optimal control captures the fundamental difference between suppression and mitigation, exhibiting a phase transition between solutions that avoid diffusion of the disease, and solutions that steer the system to herd immunity while maintaining the number of infected low to mitigate the negative effects of hospital congestion. Our second main contribution is to consider two simple feedback policies, where lockdown levels are selected to either stabilize the reproduction number $R(t)$ (i.e., reduce the number of secondary infections to  ``flatten the curve'') or the fraction of infected (i.e., containing the number of hospitalizations). We note that the policy stabilizing $i(t)$ resembles the optimal control in \cite{miclo2020optimal}, with the only difference that in \cite{miclo2020optimal} the threshold depends on hospitals capacity, while in our formulation the threshold is a free parameter that can be optimized to reduce the cost. We analytically compute the exact cost of these policies, and provide an upper bound to the performance gap between the two policies and the optimal control. In particular, we show that if $T$ is infinite, controlling $i(t)$ is better than controlling $R(t)$ and is close to optimal. Together with its close-to-optimal performance, the policy that controls the number of infected is in feedback form and easy to interpret, in contrast with numerical solutions of optimal control. As a corollary, our result shows that the spectral control approach  (see e.g., \cite{nowzari2016analysis, birge2020controlling}), where typically one aims at selecting the minimal intervention guaranteeing $R(t)<1$ is not optimal in the long run when i) the cost of the intervention has to be paid repeatedly in time (e.g., lockdown are different than vaccinations since if lockdown are interrupted the reproduction number will immediately go back to above 1, while after a vaccination campaign,  assuming permanent immunization, the  system will be stable forever on) and ii)  the horizon is long or the economic cost associated with the intervention is comparable with the epidemic cost. We remark that our theoretical analysis is restricted to the case of infinite time horizon and does not consider pharmaceutical interventions that may be implemented when cures and vaccines become available. Furthermore, we assume that individuals who contracted the infection become immune from there one. We emphasize that our purpose is not to suggest specific control strategies for the COVID-19 pandemic, but to provide general insights and mathematical contribution on the control of epidemics.%%%%%%%%%%%%%%%%%%%%%%%%%%%%%%%%%%%%%%%%%%%%%%%%%%%%%%%%%%%%%%%%%%%%%%%%%%%%%%%%
\section{Model and problem formulation}
\label{sec:model}
To capture the effect of lockdown interventions, we consider the following simple SIR model introduced in \cite{alvarez2020simple}:
\begin{equation}\label{eq:sir}
	\begin{cases}
		\dot{s}=-\beta {(1-\theta L)^2}si\\
		\dot{i}=\beta{(1-\theta L)^2} si-\gamma i\\
		\dot{r}=\gamma i,
	\end{cases}
\end{equation}
where $s/i/r$ is the fraction of susceptible/infected/recovered agents, with initial condition $s_0/i_0/r_0$, $\beta$ is the infection rate, $\gamma$ is the recovery rate, $L \in [0,\overline{L}]$ is the percentage of population in lockdown and $\theta$ is the effectiveness of the lockdown (i.e. even if $L=1$ only $\theta$ percent of the population actually complies with the lockdown) and the dependence of $s, i, r, L$ on $t$ is omitted for convenience of notation. While the transmission mechanisms of the SARS-CoV-2 are very complicated (see e.g. \cite{sun2021transmission}), we assume that the lockdown $L$ enters quadratically in the dynamics, because people get the disease through pairwise interactions, (see \cite{alvarez2020simple,acemoglu2020multi}). When considering lockdown policies, one needs to consider not only the epidemics dynamics, but also the economic cost associated with such interventions. To this end, the cost of the lockdown policy can be divided into an \textit{economic cost} (due to loss of workforce because of the lockdown) which we model as
$ cost_{eco}(t)=L(t)$,
and an  \textit{epidemic cost} (related to number of fatalities) which following \cite{alvarez2020simple} and \cite{acemoglu2020multi} we model as 
$ cost_{epi}(t)= \kappa \gamma_d(i(t))i(t),$
where $\gamma_d(i(t))i(t)$ is the mortality rate and $\kappa$ is the trade-off between economic and epidemic cost. We assume that the mortality  $\gamma_d(i)=\gamma_0+\gamma_1 i$ is increasing in $i$ to take into account hospitals congestion, as proposed in \cite{alvarez2020simple,acemoglu2020multi}.
\subsection*{Reproduction number}
The behaviour of the SIR dynamics is characterized by the reproduction number $R(t):=\beta s(t) /\gamma$, which is the number of secondary infections that an infected person produces on average. If $R(t)<1$, the epidemic is stable in the sense that the number of infected decreases indefinitely, while if $R(t)>1$ the number of infected increases until the first time $t^*$ such that $R(t^*)=1$. In presence of control, i.e., when $L(t)>0$, the infection rate and therefore the reproduction number are reduced. We let $R_L(t):=\beta(1-\theta L)^2/\gamma$ denote the \emph{controlled reproduction number} and $R(t)$ denote the \emph{uncontrolled reproduction number}.
%{\color{magenta}Quickly introduce, realte to stability and discuss controlled and uncontrolled reporduction number?}
\subsection*{Optimal control}
Let $c(t):=cost_{eco}(t)+cost_{epi}(t)$.
The objective of the planner is to find the optimal control $L^*(t)$ solving
\begin{equation}\label{eq:cost}
	\begin{aligned}
	L^* = \underset{L:[0,T] \to [0,\overline{L}]}{\arg\min} & \quad  \int_0^T c(t) dt \\
	\textup{s.t.} & \quad \textup{dynamics in }\eqref{eq:sir},
	\end{aligned}
\end{equation}
where $T$ is the time horizon, %after which for simplicity we assume that both a vaccine and a cure arrive.
%This is an optimal control problem, and to best of our knowledge no analytical solutions to this and similar problems, i.e., with non linear cost in $i$, are provided in literature. However
We start by noting that, using similar techniques as in \cite{behncke2000optimal}, one can show that the optimal lockdown $L^*(t)$ solution to the optimal control problem \eqref{eq:cost} is an increasing function of the force of infection $f:=\beta s i$. It is remarkable to note that, although the stability of the system does not depend on $i$, the lockdown policy does, through $f$. Thus, even if the system is highly unstable (i.e., %the {\color{magenta} shoudl this be the controlled?} 
$R(t)=\beta s(t)/\gamma \gg 1$) if $i(t) \rightarrow 0$ then $L^*(t) \rightarrow 0$, which means that the optimal lockdown lets the infected grow when the number of infected is low. %Based on this observation our main objective is to compare the performance of two simple feedback policies: one based on controlling the reproduction number $R(t)$ and one based on controlling the fraction of infected $i(t)$, whose connection with the force of infection will be made clear in the next section.
Besides this characterization, an analytic solution to \eqref{eq:cost} remains open. Numerical solutions to optimal control problems such as \eqref{eq:cost} have been investigated, but typically suffer from lack of interpretability and robustness guarantees, as they are not in feedback form, i.e., the policy does not depend explicitly on the state. As an alternative, in the next section we propose two simple feedback policies based on controlling the reproduction number (which is typically considered in spectral control \cite{nowzari2016analysis}) or controlling the total number of infected (to reduce congestion of the health system). 
These strategies, or their combination, have informed the response to COVID-19 of governments and public/private entities (e.g., Cornell University \cite{cornell}, USA \cite{USA} and Italy \cite{Italy}).

 %strategies adopted by governments for COVID-19 \textcolor{red}{(literature? If not, say that policy are based on spectral control and Weibull, and that Italy uses both indicators (and modify also in abstract))}.

\section{Two feedback policies}\label{sec:heuristic}
We here describe the two feedback policies. Their performance are studied in Section \ref{sec:results}. To simplify the exposition we set $\theta=1$ and $\overline{L}=1$, but  all results can be  generalized.
\subsection{Policy A: Controlling reproduction number} 
This policy aims at controlling the reproduction number, as typically considered in spectral control (see e.g. \cite{nowzari2016analysis}). Given a target reproduction number $\rho$, if the uncontrolled reproduction number $R(t)$ is greater than $\rho$, i.e. $\beta s(t) /\gamma > \rho$, a lockdown level $L_\rho(t)>0$ such that $R_L(t)=\rho$ is imposed, leading to $L_\rho$ satisfying 
	\begin{equation}\label{eq:r0_const}
		R_L(t)= \frac{\beta}{\gamma}(1-L_\rho(t))^2s(t) =\rho.
	\end{equation}
If instead $R(t)\le \rho$, then $L_\rho(t)=0$.\\
   We can divide the dynamics in two phases:
	\begin{enumerate}
		\item Phase I: in $[0,t^*)$ with $t^*$ such that $s(t^*)=\gamma \rho / \beta$. In this phase $\beta  s /\gamma\ge \rho $, $L_\rho$ is as in \eqref{eq:r0_const} and the dynamics are 
		\begin{equation*}
			\label{eq:part1rho}
				\dot{s}=-\rho \gamma i, \quad
				\dot{i}=\rho \gamma i -\gamma i = (\rho-1)\gamma i, \quad
				\dot{r}=\gamma i.
		\end{equation*}
		\item Phase II: in $[t^*,T]$. Since $R(t)\le\rho$, $L_\rho=0$, the dynamics is uncontrolled and given by
	\begin{equation*}
			\dot{s}=-\beta  si, \quad
			\dot{i}=\beta si -\gamma i, \quad
			\dot{r}=\gamma i.
	\end{equation*}
	\end{enumerate}
	Fig. \ref{fig:control_R0} illustrates the dynamics under Policy A.
	Notice that Phase I may not exist if $\beta  s_0 /\gamma \le \rho$, and Phase II may not exist if $\beta s(T)/\gamma>\rho$ (this scenario typically occurs when $\rho<1$ because the infected decay exponentially and the disease does not spread). Policy A is parametric in $\rho$, thus a policy-maker that adopts this policy should select $\rho^*$ to minimize the total cost (see Fig. \ref{fig:control_R0}), leading to
	\begin{equation*}\label{eq:r}
	\begin{aligned}
		\rho^*:=\arg\min_\rho & \quad \int_0^T c(t) dt \\
		s.t & \quad  \quad \eqref{eq:sir} \textup{ with } L(t)=L_\rho(t).
	\end{aligned}
	\end{equation*}
\begin{figure}
	\centering
\includegraphics[width=4.2cm]{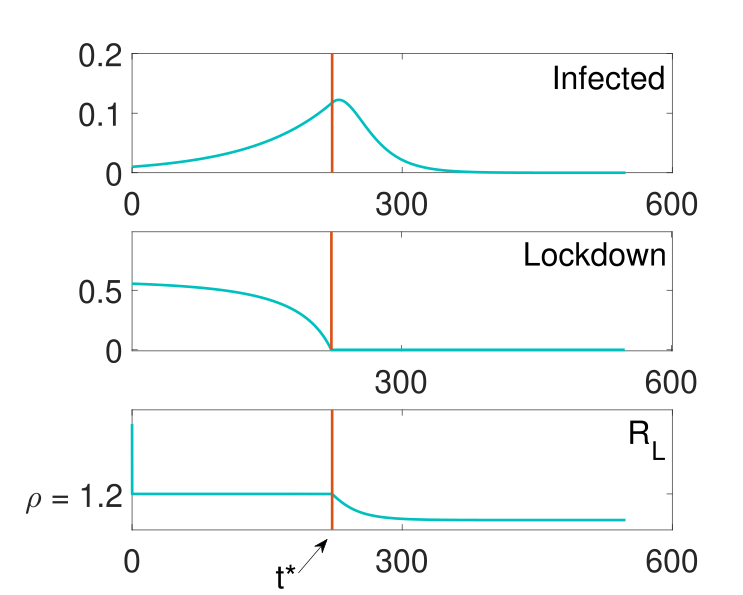}
\includegraphics[width=4.2cm]{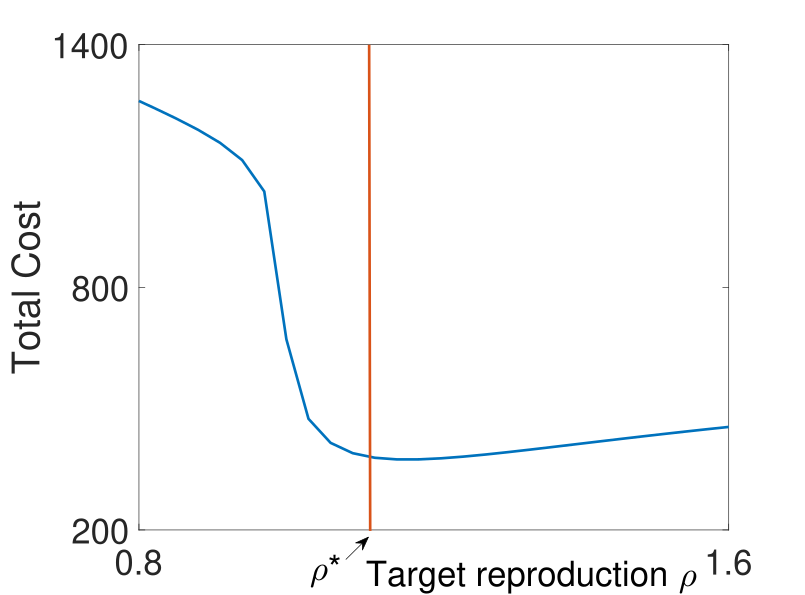}
%\\
%\includegraphics[width=6cm]{figures/3f-eps-converted-to.pdf}
\caption{Policy A. \emph{Left}: Dynamics with $\rho=1.2$. Phase I ($L_\rho(t)>0$) lasts until $\beta s /\gamma = \rho$ (red line); then, the trajectory becomes uncontrolled (Phase II). Infected grows until the reproduction number becomes less than 1. The unit of time is one day. \emph{Right}: optimization over $\rho$. The parameters used in this simulation, and are $\beta=0.2$, $\gamma=1/18$, $\gamma_0=5.6 \cdot 10^{-4}$, $\gamma_1=5.6 \cdot 10^{-3}$ based on the COVID-19 pandemic, see \cite{alvarez2020simple, acemoglu2020multi}), $s_0=0.98$, $i_0=0.01$, $r_0=0.01$, $\kappa=40\cdot365$, $T=5$ years. For the used parameters $\rho^*>1$, meaning that the total cost is minimized with a mitigation policy that lets the infection spread in a controlled way. \label{fig:control_R0}}%\emph{Below}: trajectory with $\rho=0.8$. Since the reproduction number is less than 1, infected decay exponentially and Part II of the dynamics never starts.}
\end{figure}
\subsection{Policy B: controlling fraction of infected}
As anticipated in Section \ref{sec:model}, optimality conditions suggest that the optimal quantity to be controlled is the force of infection \cite{behncke2000optimal}. Stabilizing this quantity, i.e., imposing lockdown $L_f(t)$ such that that the $\beta(1- L_f(t))^2 s(t) i(t) = \phi$ yields the dynamics
\begin{equation*}
	\label{eq:dyn_force}
	\dot{s}=-\phi, \quad
	\dot{i}=\phi -\gamma i, \quad
	\dot{r}=\gamma i,
\end{equation*}
which stabilizes asymptotically $i(t)$ to $\phi/\gamma$. Based on this observation, instead of controlling the force of infection, we here propose a policy that stabilizes the number of infected $i(t)$ (as the optimal policy in \cite{miclo2020optimal}), whose meaning is more intuitive. Given a target value $\iota$, policy B works as follows: if $i_0<\iota$, $L_\iota(t)=0$ until the first time $\tau_1$ such that $i(\tau_1)=\iota$; if $i_0>\iota$, $L_\iota(t)=1$ until $\tau_1$; then, impose $L_\iota(t)$ in such a way that $\dot{i}(t)=0$, hence $i(t)=\iota$, until time $\tau_2$ such that herd immunity is reached, i.e., $s(\tau_2)=\gamma/\beta$; finally, the lockdown is released, i.e., $L_\iota(t)=0$ (see Fig. \ref{fig:control_i}). Given $\iota$, we divide the trajectory in three phases:
\begin{enumerate}
	\item Phase I in $[0,\tau_1)$, with $\tau_1$ being the smallest time such that $i(\tau_1)=\iota$. If $i_0 < \iota$  $L_\iota(t)=0$ and the dynamics is
%	\begin{equation*}
%	\begin{cases}
%		\dot{s}=-\beta s i \\
%		\dot{i}=\beta s i -\gamma i  \\
%		\dot{r}=\gamma i,
%	\end{cases}
%\end{equation*}
	\begin{equation*}
		\dot{s}=-\beta s i, \quad
		\dot{i}=\beta s i -\gamma i, \quad
		\dot{r}=\gamma i,
\end{equation*}
If $i_0 > \iota$, $L_\iota(t)=1$, with dynamics
%		\begin{equation*}
%		\begin{cases}
%		\dot{s}=0 \\
%		\dot{i}= -\gamma i  \\
%		\dot{r}=\gamma i.
%	\end{cases}
%\end{equation*}
		\begin{equation*}
		\dot{s}=0, \quad
		\dot{i}= -\gamma i, \quad
		\dot{r}=\gamma i.
\end{equation*}
	\item Phase II in $[\tau_1,\tau_2)$, with $\tau_2$ such that $s(\tau_2)=\gamma/\beta$. In this phase $i(t)=\iota$, and the dynamics is
%	\begin{equation*}
%		\begin{cases}
%		\dot{s}=-\gamma \iota \\
%		\dot{i}=0 \\
%		\dot{r}=\gamma \iota,
%	\end{cases}
%\end{equation*}
	\begin{equation*}
		\dot{s}=-\gamma \iota, \quad
		\dot{i}=0, \quad
		\dot{r}=\gamma \iota.
\end{equation*}
with 
\begin{equation*}
	L_\iota(t) \ \text{s.t.}\ \beta s(t) (1-L_\iota(t))^2 = \gamma.
\end{equation*}
	\item Phase III in $[\tau_2,T]$: the dynamics is uncontrolled and given by
	\begin{equation*}
		\dot{s}=-\beta s i, \quad
		\dot{i}=\beta s i -\gamma i,  \quad
		\dot{r}=\gamma i.
	\end{equation*}
\end{enumerate}
Notice that in Phase III $\dot{i}(t)<0$ even without control because herd immunity is reached.
In contrast with \cite{miclo2020optimal}, where the value of $\iota$ depends on the ICU capacity and the policy-maker has to choose the optimal policy complying with the constraint $i(t)\le \iota$, here for any value of $\iota$ the policy is given, and the policy-maker should select $\iota^*$ to minimize the total cost \eqref{eq:cost} (see Fig. \ref{fig:control_i}), leading to
\begin{equation*}\label{eq:i}
	\begin{aligned}
		\iota^*:=\arg\min_\iota & \quad \int_0^T c(t) dt \\
		s.t & \quad  \quad \eqref{eq:sir} \textup{ with } L(t)=L_\iota(t).
	\end{aligned}
\end{equation*}
%\begin{figure}
%	\centering
%	\includegraphics[width=6cm]{figures/cost_vs_imax_f-eps-converted-to.pdf}
%	\caption{Optimization over $\iota$. \label{fig:opt_i}}
%\end{figure}
\begin{figure}
	\centering
	\includegraphics[width=4.2cm]{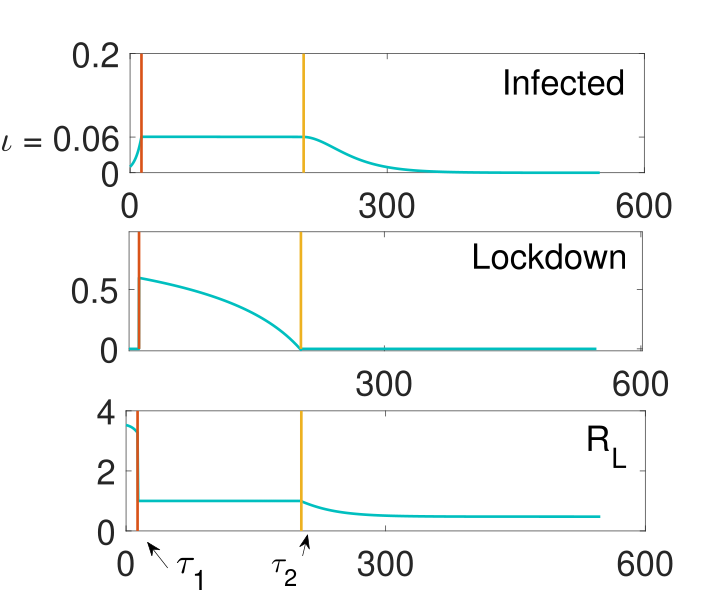}
	\includegraphics[width=4.2cm]{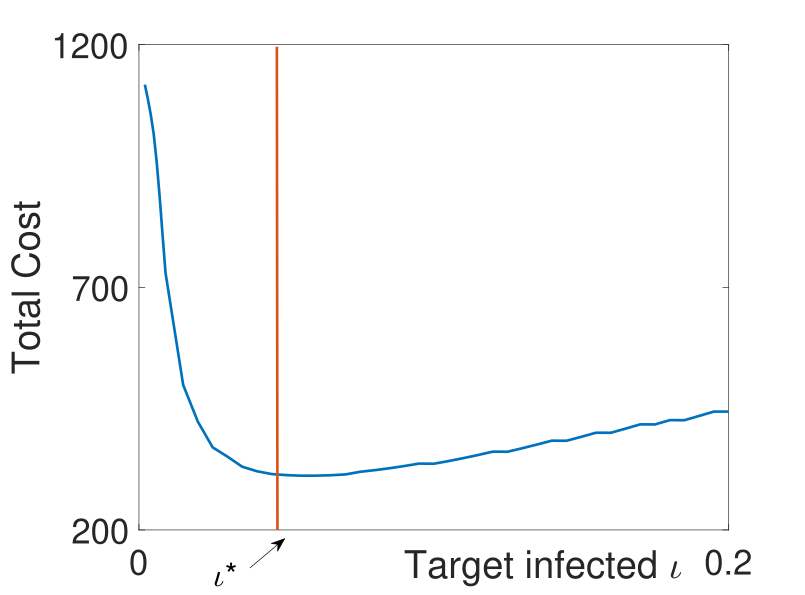}
	\caption{Policy B. \emph{Left}: Dynamics with $\iota=0.06$. In Phase I the trajectory is uncontrolled ($L(t)>0$) and the fraction of infected grows (Phase I); when $i(t)$ reaches the threshold $\iota$, lockdown $L(t)$ is imposed in such a way to maintain the fraction of infected constant (Phase II). When herd immunity is reached, the lockdown is released and infected start decreasing (Phase III). \emph{Right}: Optimization over $\iota$. The parameters are as in Fig. \ref{fig:control_R0}. \label{fig:control_i}}
\end{figure}
\section{Results}\label{sec:results}
This section is divided in two parts. In the first part, some observations on numerical solutions are presented. In the second part, we provide an upper bound on the performance gap between the feedback policies presented in Section \ref{sec:heuristic} and the optimal control \eqref{eq:cost}, and show that, some under assumptions, controlling the fraction of infected is close to optimal and outperforms the policy based on controlling the reproduction number.
\subsection{Numerical results}
\begin{figure}
	\centering
	\includegraphics[width=4.2cm]{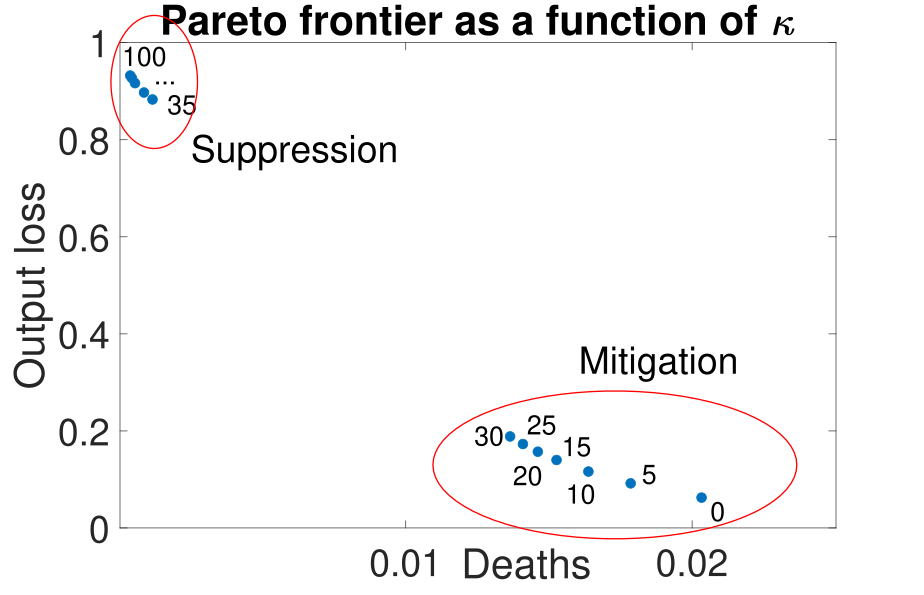}
	\includegraphics[width=4.2cm]{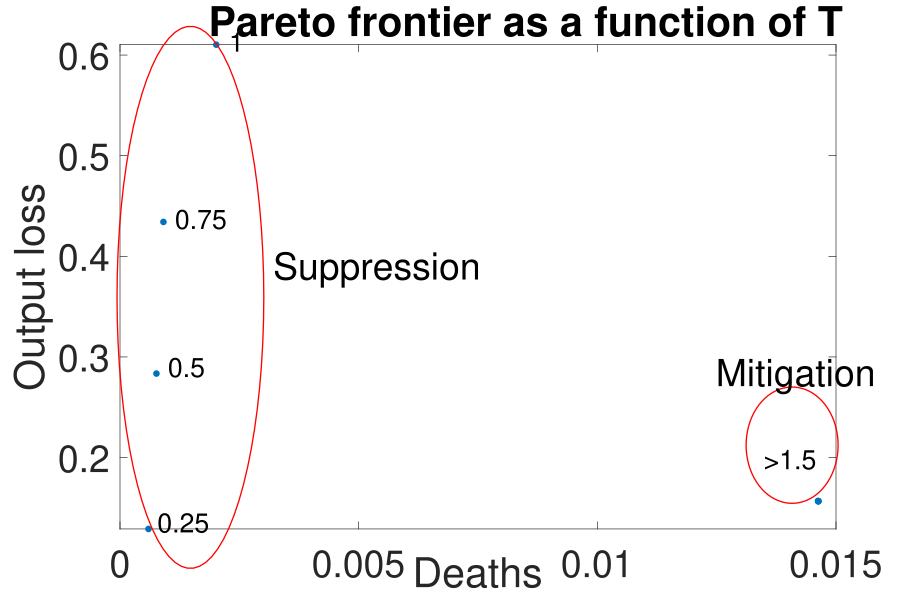}
	\caption{\emph{Left}: Economic loss, defined as $\frac{\int_0^T L(t)dt}{T} $, vs epidemic cost, i.e., fraction of population died because of the disease under the optimal control, as the trade-off between economic and epidemics cost $\kappa$ vary ($T=1.5$ yrs). \emph{Right}: same with $T$ varying, ($\kappa=20 \cdot 365$), and economic cost defined as $\frac{\int_0^T L(t)dt}{1.5 \text{yrs}}$. Other parameters are as in Fig. \ref{fig:control_R0}.}
		\label{fig:phase}
\end{figure}
Our first result is that numerical solutions of the optimal control exhibit a phase transition from suppression to mitigation, as described in \cite{ferguson2020report}, as $T$ increases and $\kappa$ decreases. In the left panel of Fig. \ref{fig:phase}, economic and epidemic cost as their trade-off $\kappa$ vary are computed, while other parameters (including $T$) are fixed. There is a clear phase transition: when $\kappa$ is small, the optimal policy lets the epidemic spread (high number of deaths, small economic cost), whereas when $\kappa$ is large optimal policy maintains the lockdown until the horizon (high economic cost and small epidemic cost). Solutions in between these two strategies are never optimal. The right plot shows that a similar phase transition occurs with $\kappa$ constant and varying $T$. Despite the economic and epidemic cost exhibiting a phase transition, we emphasize that the total cost is continuous in the parameters.
\begin{remark}
	We also observe that similar phase transition occurs for Policies A and B, namely if $\kappa$ is large and $T$ is small, the optimal reproduction number for Policy A satisfies $\rho^*<1$, implying that $i(t)$ decreases, and under Policy B the optimal $\iota^*$ is very small, whereas, if $\kappa$ is small and $T$ is large, the optimal reproduction number satisfies $\rho^*>1$, the optimal $\iota^*$ is much larger, and the epidemic cost is in general comparable with the economic cost.
\end{remark}\medskip
This confirms our intuition that solutions based on spectral control problems, that minimize statically the cost of intervention subject to the constraint $R_0<1$, are suboptimal from an optimal control perspective when $T$ is large or  $\kappa$ is not too high, since stabilizing the system indefinitely leads to a prohibitively
 high economic cost. As shown in Fig. \ref{fig:suppression_mitigation}, in the range of parameters that make the optimal control select suppression strategies, the optimal control and the feedback policies achieve similar performances, whereas when the optimal control selects mitigation, Policy B is close to optimal and outperforms Policy A. We next investigate these observations.
\subsection{Theoretical results}
\label{sec:results}

\begin{table*}	\caption{Analytic expressions of the cost under the two heuristics}
\vspace{-0.3cm}\hrule
	\begin{equation}\label{eq:t*}
			\begin{aligned}
			&C_A(\rho)%&= \int_0^\infty [wL + M \gamma_q i^2  + M\gamma_d i]dt \\
			=(t^*-T_1)+\frac{\kappa \gamma_0}{\gamma} (r(\infty)-r_0)
				+\frac{\kappa \gamma_1 }{2(\rho-1)\gamma} \left(i(t^*)^2 -i(0)^2\right)
				+\kappa\gamma_1\bigg(\gamma\frac{[\ln(s(t^*))]^2-[\ln(s(\infty))]^2}{2\beta^2}-\frac{s(t^*)-s(\infty)}{\beta} +\big(1-r(t^*)-\frac{\gamma}{\beta}\ln (s(t^*))\big)\frac{r(\infty)-r(t^*)}{\gamma}\bigg),\\
			&	\textup{with}\quad  t^*=\frac{ ln\left( 1 - \frac{\rho\gamma-\beta s_0 }{\beta\nu i_0}\right)}{(\rho-1)\gamma}, \ 
	\nu = \frac{\rho}{\rho+1}, \ \tilde{\nu}=(s_0+\nu i_0), \
	s(t^*)=\frac{\gamma \rho}{\beta},  \ r(t^*)=1-\frac{s(t^*) + \tilde{\nu}}{\nu},\\&\qquad
	T_1=2\sqrt{\frac{\rho }{\beta \tilde{\nu}(\rho-1)^2\gamma }}  { \left(     \at\left(\sqrt{\frac{s_0}{\tilde{\nu}}}\right) - \at\left(\sqrt{\frac{ \rho\gamma}{\beta\tilde{\nu}}}\right)    \right) }, \
	r(\infty)=1-s(\infty)=1-s(t^*) e^{-\frac{\beta}{\gamma} (r(\infty)-r(t^*))}
			\end{aligned}
		\end{equation}
		\hrule
		\vspace{0.2cm}
		\begin{equation}\label{eq:cost_B}
		\begin{aligned}
		&C_B(\iota)%&= \int_0^T [wL + M \gamma_q i^2  + M\gamma_d i]dt \\
		=\kappa\gamma_1\Bigg(\gamma\frac{[\ln s_0]^2-[\ln s(\tau_1)]^2}{2\beta^2}-\frac{s_0-s(\tau_1)}{\beta} +\big(s_0+i_0-\frac{\gamma}{\beta}\ln s_0\big)\frac{i_0+s_0-\iota-s(\tau_1)}{\gamma}+(\tau_2-\tau_1)\iota^2+\gamma\frac{[\ln \frac{\gamma}{\beta}]^2-[\ln s(\infty)]^2}{2\beta^2}-\frac{\frac{\gamma}{\beta}-s(\infty)}{\beta} +\\
		&+\big(\frac{\gamma}{\beta}+\iota-\frac{\gamma}{\beta}\ln \frac{\gamma}{\beta}\big)\frac{r(\infty)-1+\iota+\frac{\gamma}{\beta}}{\gamma}\Bigg) + \frac{1}{\iota}\bigg(\sqrt{\frac{s(\tau_1)}{\gamma}}-\sqrt{\frac{1}{\beta}}\bigg)^2+\frac{\kappa\gamma_0}{\gamma} (r(\infty)-r_0)\\
		&\textup{where } s(\tau_1) \textup{ and } r(\infty) \textup{ are the unique feasible solutions of}\\& \quad \ln\Big(\frac{s(\tau_1)}{s_0}\Big)=\frac{1}{\gamma}[s(\tau_1)-s_0+\iota-i_0], \
		r(\infty)=1-\frac{\gamma}{\beta} e^{-\frac{\beta}{\gamma} (r(\infty)-1+\iota+\frac{\gamma}{\beta})}, \ \tau_2-\tau_1=\frac{s(\tau_1)-\gamma/\beta}{\gamma \iota}, \
		s(\infty)=1-r(\infty).
	\end{aligned}
		\end{equation}
		\hrule
\end{table*}
\begin{figure}
	\centering
	\includegraphics[width=4.2cm]{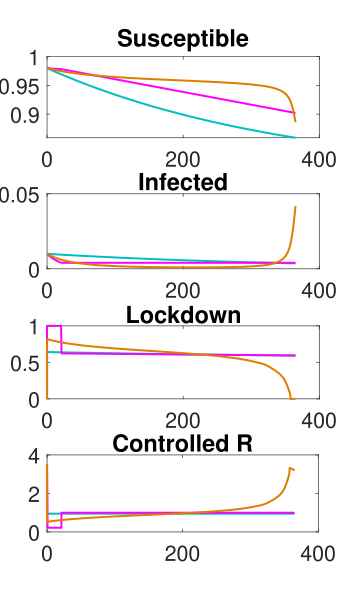}
	\includegraphics[width=4.2cm]{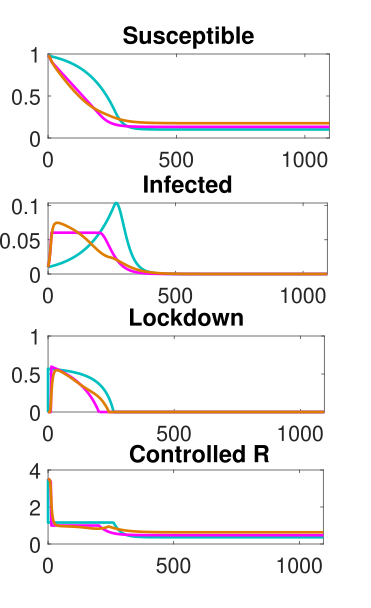}\\
	\includegraphics[width=8cm]{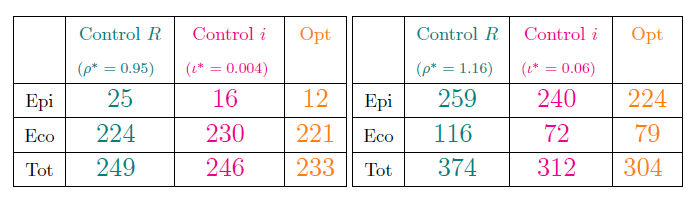}
	\caption{\emph{Left}: $T=1 yr$, optimal control (orange) and feedback policies (policy A in blue, policy B in purple) select suppression strategy. %The main difference between the policies is that optimal control releases the lockdown close to the horizon. 
	Performances of the three strategies are comparable, and economic cost is much larger than epidemic cost. \emph{Right}: $T=3 yrs$, %Suppressing the disease with such a long horizon is too costly from economic perspective. 
	optimal control and feedback policies select mitigation. The dynamics under policy B is similar to dynamics under the optimal control, and Policy B outperforms Policy A. In this simulation $\kappa=1.5$, and the other parameters are as in Fig. \ref{fig:control_R0}. \label{fig:suppression_mitigation}}
\end{figure}
To compare the cost achieved by the feedback policies with the optimal cost, we work under the assumption that $T=+\infty$, which has two main implications:
\begin{enumerate} 
	\item there exists $t^*$ such that $s(t^*)=\gamma/\beta$ under optimal control, that is, herd immunity is reached; indeed, if it were not the case, the lockdown could never be totally released and the economic cost would explode; 
	\item the fraction of infected $i(T)$ vanishes (see e.g. \cite{hethcote2000mathematics}).
\end{enumerate}
In other words, with infinite $T$ there is no phase transition between suppression and mitigation (unless $\kappa=0$), since avoiding the spread of the infection leads to an infinite economic cost and is therefore unfeasible.
\begin{remark}
We observed from simulations that above a certain value of $T$, the optimal policy does not depend on $T$. Indeed, if the epidemic reaches herd immunity, the infected naturally asymptotically decrease to $0$ and both epidemic and economic cost tend to vanish. Hence our results provide insight not only for the infinite horizon case, but for any large $T$.% We do not make use directly of the assumption $T=+\infty$. Instead, we use its implications 1) and 2), that are verified for $T$ sufficiently high.
\end{remark}\medskip
Before establishing our main results, we introduce a technical lemma on SIR dynamics which is needed for the proofs and may be of separate interest.
\begin{lemma}[Uncontrolled dynamics]
	\label{prp:nc_square}
	Let $a \le b$, and let $L(t)=0$ for every $t \in [a,b]$.
	Then,
	\begin{equation*}
		\begin{aligned}
			\int_a^b i^2dt
			&=\gamma\frac{[\ln(s(a))]^2-[\ln(s(b))]^2}{2\beta^2}-\frac{s(a)-s(b)}{\beta} +\\ &+\big(s(a)+i(a)-\frac{\gamma}{\beta}\ln (s(a))\big)\frac{r(b)-r(a)}{\gamma}.
		\end{aligned}
		\label{eq:i^2_nc}
	\end{equation*}
\end{lemma}\medskip
%\begin{proof}
%	Note that $s(t)+i(t)-\frac{\gamma}{\beta}s(t)$ is a constant of motion if $L(t)=0$, hence
%	\begin{equation}
%		\label{eq:constant}
%		s(a)+i(a)-\frac{\gamma}{\beta}s(a) = s(b)+i(b)-\frac{\gamma}{\beta}s(b).
%	\end{equation}
%By \eqref{eq:constant}, omitting the dependence on $t$,
%	\begin{equation}
%		\begin{aligned}
%			&\int_a^b i^2dt=\int_a^b i\big(\frac{\gamma}{\beta}\ln (s)-\frac{\gamma}{\beta}\ln s(a)-s+s(a)+i(a)\big)dt\\
%			=&\frac{\gamma}{\beta}\int_a^b i \ln (s) dt-\int_a^b is dt + \big(s(a)+i(a)-\frac{\gamma}{\beta}\ln s(a)\big)\int_a^b i dt.
%		\end{aligned}
%		\label{eq:sum_pieces}
%	\end{equation}
%	From $\dot{s}=-\beta s i$, the first integral gets:
%	\begin{equation}
%		\begin{aligned}
%			&\int_a^b i \ln(s) dt= -\frac{1}{\beta}\int_a^b \frac{\dot{s}\ln(s)}{s}dt\\
%			=&-\frac{1}{2\beta}\int_a^b\frac{d}{dt}[\ln(s)]^2dt=\frac{[\ln s(a)]^2-[\ln s(b)]^2}{2\beta}.
%			\label{eq:piece1}
%		\end{aligned}
%	\end{equation}
%	The second and the third integral are
%	\begin{equation}
%		\label{eq:piece2}
%		\begin{aligned}
%			&\int_a^b is dt=-\frac{1}{\beta}\int_a^b \frac{\dot{s}}{s}sdt=-\frac{1}{\beta}\int_a^b \dot{s}dt=\frac{s(a)-s(b)}{\beta},\\
%			&\int_a^b i(t)dt=\int_a^b\frac{\dot{r}(t)}{\gamma}dt=\frac{r(b)-r(a)}{\gamma}.
%		\end{aligned}
%	\end{equation}
%	The statement follows by plugging \eqref{eq:piece1} and \eqref{eq:piece2} in \eqref{eq:sum_pieces}.
%\end{proof}\medskip
\begin{proof}
	It can be shown that $s(t)+i(t)-\frac{\gamma}{\beta}s(t)$ is constant in time if $L(t)=0$, hence
	\begin{equation}
		\label{eq:constant}
		s(a)+i(a)-\frac{\gamma}{\beta}s(a) = s(b)+i(b)-\frac{\gamma}{\beta}s(b).
	\end{equation}
	By \eqref{eq:constant}, omitting the dependence on $t$,
	\begin{equation}
		\begin{aligned}
			&\int_a^b i^2dt=\int_a^b i\big(\frac{\gamma}{\beta}\ln (s)-\frac{\gamma}{\beta}\ln s(a)-s+s(a)+i(a)\big)dt\\
			=&\frac{\gamma}{\beta}\int_a^b i \ln (s) dt-\int_a^b is dt + \big(s(a)+i(a)-\frac{\gamma}{\beta}\ln s(a)\big)\int_a^b i dt.
		\end{aligned}
		\label{eq:sum_pieces}
	\end{equation}
	From $\dot{s}=-\beta s i$, the first integral gets:
	\begin{equation}
		\begin{aligned}
			&\int_a^b i \ln(s) dt= -\frac{1}{\beta}\int_a^b \frac{\dot{s}\ln(s)}{s}dt\\
			=&-\frac{1}{2\beta}\int_a^b\frac{d}{dt}[\ln(s)]^2dt=\frac{[\ln s(a)]^2-[\ln s(b)]^2}{2\beta}.
			\label{eq:piece1}
		\end{aligned}
	\end{equation}
	The second and the third integral are
	\begin{equation}
		\label{eq:piece2}
		\begin{aligned}
			&\int_a^b is dt=-\frac{1}{\beta}\int_a^b \frac{\dot{s}}{s}sdt=-\frac{1}{\beta}\int_a^b \dot{s}dt=\frac{s(a)-s(b)}{\beta},\\
			&\int_a^b i(t)dt=\int_a^b\frac{\dot{r}(t)}{\gamma}dt=\frac{r(b)-r(a)}{\gamma}.
		\end{aligned}
	\end{equation}
	The statement follows by plugging \eqref{eq:piece1} and \eqref{eq:piece2} in \eqref{eq:sum_pieces}.
\end{proof}\medskip
Let $C_{A}(\rho)$ and $C_B(\iota)$ denote the cost of policies $A$ and $B$ for a given value of $\rho$ and $\iota$, respectively. In the next two propositions we derive analytical expressions for $C_{A}(\rho)$ and $C_B(\iota)$. For the sake of readability, for the two policies we work respectively under the assumption that $\iota \ge i_0$ and $\rho > 1$, but similar expressions hold in case $\iota<i_0$ and $\rho \le 1$. 
\begin{proposition}
	\label{prpA}
	Let $T = +\infty$. Then, $C_A(\rho)$ is as in \eqref{eq:t*}.
\end{proposition}
\begin{proof}
	See the Appendix.
\end{proof}
\begin{proposition}
	\label{prpB}
Let $T = +\infty$. Then, $C_{B}(\iota)$ is as in  \eqref{eq:cost_B}.
\end{proposition}
\begin{proof}
	See the Appendix.
\end{proof}
%The idea of the proof is similar to Proposition \ref{prpA}.
The idea for the proofs is to split the cost in three terms:
\begin{enumerate}
	\item economic cost: $\int_0^T L dt$;
	\item infected cost:
	$\int_0^T i dt  $;
	\item infected squared cost:
	$\int_0^T i^2 dt  $.
\end{enumerate}
The terms may be computed exactly by using Lemma \ref{prp:nc_square} and by using the fact that in presence of control the dynamics is linear.
In the next proposition we give a lower bound on the optimal cost, which will be used to compare Policy A and B with the optimal control. For technical reasons we restrict our analysis to policies that do not put control after herd immunity. This is a natural assumption, and numerical solutions show that this is indeed the case for most values of parameters.
	\begin{proposition}
		\label{prp:lb}
	Let $T=+\infty$, and let $t^*$ be the smallest time such that $s(t^*)\le\gamma/\beta$ (herd immunity). The cost of any control such that $L(t)=0$ for every $t \in [t^*,+\infty)$ is lower bounded by
	\begin{equation*}
		\label{eq:lower_bound}
		C^* := \underset{0\le i^* \le i_{max}}\min U(i^*),
	\end{equation*}
	where
	$
	 U(i^*)=2 {U_{12}^I(i^*)}+U_{22}^{II}(i^*)+U_{3}(i^*)
	 $
	 and
	 \begin{equation*}\footnotesize
		\begin{aligned}\footnotesize
			U_{12}^I(i^*)&= \epsilon \sqrt{\kappa\gamma_1} \left[ \frac{ -i^*+ i_0-s({t^*})+s_0}{\gamma} +\frac{1}{\beta} (\ln{s(t^*})- \ln s_0)\right],\\	
		U_{22}^{II}(i^*) &=
			\frac{\gamma}{\beta}\frac{[\ln{(\gamma/\beta)}]^2-[\ln(s(\infty))]^2}{2\beta}-\frac{\gamma/\beta-s(\infty)}{\beta} +\\ &+\big(\gamma/\beta+i^*-\frac{\gamma}{\beta}\ln {\gamma/\beta}\big)\frac{i^*+\gamma/\beta-s(\infty)}{\gamma},\\
			U_3(i^*)&= \frac{\kappa\gamma_0}{\gamma}(r(\infty)-r_0),
		\end{aligned}
	\end{equation*}
	with $r(\infty)$ solution of $
		r(\infty)=1-s(t^*) e^{-\frac{\beta}{\gamma} \big(r(\infty)-1+i^*+s(t^*)\big)}$, $s(\infty)=1-r(\infty)$, $i_{max}=i_0+s_0-\frac{\gamma}{\beta}\big(1-\ln\frac{\gamma}{\beta s_0}\big)$, $\epsilon = \sqrt{27/32}$.
\end{proposition}\medskip
\begin{remark}
	Notice that $U(i^*)$ depends on $i^*$ both explicitly and through $s(\infty)$ and $r(\infty)$.
\end{remark}\medskip
\begin{proof}
	See the Appendix.
\end{proof}
Apart from constants, $U_3$ is the infected cost, and $U_{22}^{II}$ is the infected squared cost after herd immunity, computed starting from the state at time $t^*$ under the assumption that there is no control after herd immunity (so that economic cost is zero). These are computed using \eqref{eq:sir} with $L(t)=0$ and Lemma \ref{prp:nc_square}, respectively. The term $U_{12}^I$ is a lower bound for the sum of economic and infected squared cost before herd immunity. \\
Let $C_{opt}$ denote the cost of the optimal control among all the controls such that after herd immunity do not put any lockdown.
We can now derive bounds between performance of our two policies and $C_{opt}$.
\begin{theorem}
	Let $T = +\infty$. Then,
	\begin{equation*}
		C_A(\rho^*)-C_{opt} \le \min_\rho \ \eqref{eq:t*} - C^*.
	\end{equation*}
\end{theorem}
\begin{proof}
The proof follows immediately from Propositions \ref{prpA} and~\ref{prp:lb}.
\end{proof}\medskip
\begin{theorem}
Let $T = +\infty$. Then,
\begin{equation*}
	C_B(\iota^*)-C_{opt} \le \min_\iota \ \eqref{eq:cost_B} - C^*.
\end{equation*}
\end{theorem}
\begin{proof}
The proof follows immediately from Propositions \ref{prpB} and \ref{prp:lb}.
\end{proof}\medskip
In Fig. \ref{fig:plot} the exact cost of the two policies, the cost under the optimal control computed numerically, and the lower bound on the optimal cost are plotted as a function of $\kappa$. Note that controlling $i(t)$ is close to the optimal control computed numerically. The theoretical analysis allows to conclude that, for the tested parameters, controlling $i(t)$ is better than controlling $R(t)$, as the distance between cost of Policies A and B is comparable with the distance between Policy B and the lower bound $C^*$. From a computational viewpoint, note that controlling $i(t)$ instead of solving the optimal control, reduces the complexity of the problem from an optimization over a set of functions to a minimization over the parameter $\iota$.
\begin{figure}
	\centering
	\includegraphics[width=6cm]{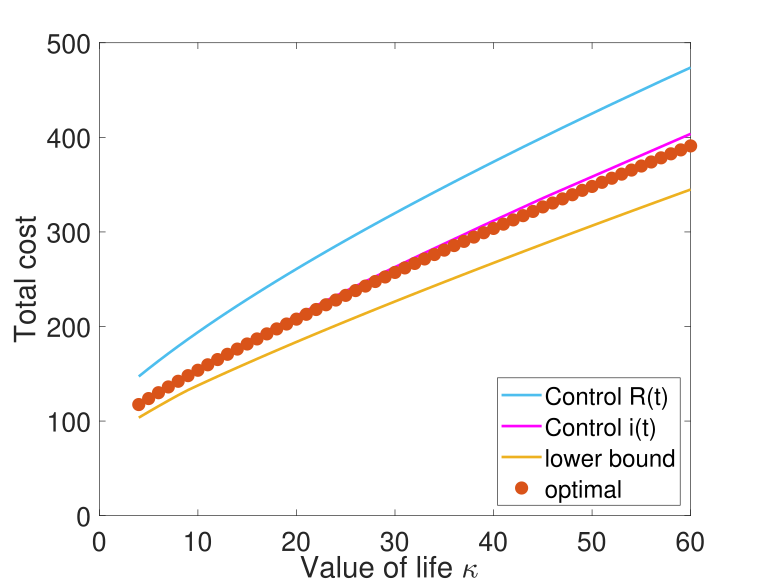}
	\caption{\label{fig:plot} Exact cost of Policies A and B with optimal $\rho^*$ and $\iota^*$ (analytical), optimal cost (numerical) and lower bound to the optimal cost (analytical) as functions of $\kappa$. We noticed that $T=5$ years is large enough to approximate infinite time horizon. Indeed, after a large enough time, both economic and epidemic cost are negligible (see Fig. \ref{fig:suppression_mitigation}). Other parameters are as in Fig. \ref{fig:control_R0}.}
\end{figure}
\section{Conclusion}
In this work we consider a SIR with lockdown, and we study the optimal lockdown as solution to an optimal control problem. We model the cost as the sum of an epidemic cost, and economic cost due to lockdowns. We first show numerically that, as the time horizon and the cost of life vary, the solutions of optimal control exhibit a phase transition from suppression strategies to mitigation strategies, as described in \cite{ferguson2020report}. Then, we introduce two simple feedback policies, one stabilizing the reproduction number, and one stabilizing the fraction of infected. We compute their exact cost, and compare analytically their performance with the optimal cost. We show that in case of infinite time horizon controlling the fraction of infected is close to optimal. Future directions include but are not limited to: investigation on where the phase transition occurs; extending the analysis to the case of finite time horizon, considering different cost functions like sigmoidal functions, and extending the analysis to network SIR, as in the numerical analysis of \cite{acemoglu2020multi}.

\section*{Acknowledgements}
%This research was carried on within the framework of the MIUR-funded {\it Progetto di Eccellenza} of the {\it Dipartimento di Scienze Matematiche G.L.~Lagrange}, Politecnico di Torino, CUP: E11G18000350001. It received partial support from the MIUR  Research Project PRIN 2017 ``Advanced Network Control of Future Smart Grids'' (http://vectors.dieti.unina.it), and by the {\it Compagnia di San Paolo} through a Joint Research Project and project ``SMaILE-- Simple Methods for Artificial Intelligence Learning and Education''.
This research was carried on within the framework of the MIUR-funded {\it Progetto di Eccellenza} of the {\it Dipartimento di Scienze Matematiche G.L. Lagrange}, Politecnico di Torino, CUP: E11G18000350001. It received partial support from the {\it Compagnia di San Paolo} through a Joint Research Project. It was also supported by C3.ai Digital Transformation Institute award.

\bibliographystyle{IEEEtran}
\bibliography{bibliography}

\section*{Appendix}
Before giving the proofs of the Propositions, we compute the asymptotic state of an uncontrolled SIR.
Using \eqref{eq:constant} with $b=+\infty$, and using the well known fact that $\lim_{t \rightarrow \infty} i(t) = 0$ \cite{hethcote2000mathematics}, which implies $r(\infty)=1-s(\infty)$, we get
\begin{equation}
	\label{eq:r_inf}
	r(\infty)=1-s(a) e^{-\frac{\beta}{\gamma} (r(\infty)-r(a))}.
\end{equation}
\subsection*{Proof of Proposition \ref{prpA}}
	Solving explicitly Phase I of the dynamics, one gets:
	\begin{equation}
		\label{eq:explicit_part1}
		\begin{aligned}
			s(t^*) &= \frac{\gamma \rho}{\beta}, \quad t^*=\frac{ ln\left( 1 - \frac{\rho\gamma-\beta s_0 }{\beta\nu i_0}\right)}{(\rho-1)\gamma},\\
			i(t^*) &= i_0 e^{2(\rho-1)t^*}. %\quad r(t^*)=1-s(t^*)-r(t^*).
		\end{aligned}
	\end{equation}
	Also,
	\begin{align*}
		s(t) &=s_0+ \int_0^t \dot{s}(\tau) d\tau = s_0-\rho \gamma  i_0 \int_0^t  e^{(\rho-1)\gamma \tau} d\tau  \\
		%& =s_0-\rho \gamma  i_0 \left[ \frac{ e^{(\rho-1)\gamma \tau}}{(\rho-1)\gamma} \right]_0^t  \\
		&= s_0+\frac{\rho   i_0}{{(\rho-1)}} ( 1- e^{(\rho-1)\gamma t}),
	\end{align*}
	which allows concluding 
	\begin{equation}
		\label{eq:st*}
	s(t^*) = s(0)+\frac{\rho   }{{\rho-1}} ( i(0)- i(t^*)), \quad
	r(t^*)= 1 - \frac{s(t^*)+\tilde{\nu}}{\nu}.
	\end{equation}
	%We write the dynamics as
	%\begin{equation*}
	%		\dot{s}=-\beta u si, \quad
	%		\dot{i}=\beta u si -\gamma i, \quad
	%		\dot{r}=\gamma i.
	%\end{equation*}
	%where 
	The cost is hence
	\begin{equation}
		\label{eq:cost_parts}
			\int_0^\infty c(t) dt
			=\int_0^\infty \left[L + \kappa\gamma_d(i)i\right]dt=\int_0^\infty \left[L + \kappa \gamma_1 i^2  + \kappa\gamma_0 i\right] dt. 
	\end{equation}
	The cost is thus composed of three parts: the economic cost, the infected square cost, and the infected cost (see Section \ref{sec:results}).
The infected cost is
\begin{equation}
	\label{eq:lin}
	\int_0^T i dt =\frac{1}{\gamma} \int_0^\infty \dot{r} dt = \frac{1}{\gamma} (r(\infty)-r_0),
\end{equation}
where $r(\infty)$ follows by plugging $a=t^*$ in \eqref{eq:r_inf} and using \eqref{eq:explicit_part1}. For the infected squared cost we separate the two phases. For Phase I:
\begin{equation}
	\label{eq:T_2}
	\begin{aligned}
		&\int_0^{t^*} i^2 dt = i_0^2  \int_0^{t^*} e^{2(\rho-1)\gamma t}  dt = \frac{i_0^2 }{2(\rho-1)\gamma} \left[e^{2(\rho-1)\gamma t} \right]_0^{t^*} \\
		=& \frac{i_0^2 }{2(\rho-1)\gamma} \left(e^{2(\rho-1)\gamma t^*} -1\right) = \frac{1 }{2(\rho-1)\gamma} \left(i(t^*)^2 -i_0^2\right).
	\end{aligned}
\end{equation}
For Phase II, $L(t)=0$, hence, by Lemma \ref{prp:nc_square},
\begin{equation}
	\begin{aligned}
		\int_{t^*}^{\infty} i^2dt
		&=\gamma\frac{[\ln(s(t^*))]^2-[\ln(s(\infty))]^2}{2\beta^2}-\frac{s(t^*)-s(\infty)}{\beta} +\\ &+\big(s(t^*)+i(t^*)-\frac{\gamma}{\beta}\ln (s(t^*))\big)\frac{r(\infty)-r(t^*)}{\gamma},
	\end{aligned}
	\label{eq:i^2_A}
\end{equation}
where $s(\infty)=1-r(\infty)$. 	We let $u:=(1-L)^2 \in [0,1]$.
	We have economic cost only in Phase I when $u=\frac{\rho \gamma}{\beta s }$. Let
	\begin{equation*}\footnotesize
\begin{aligned}\footnotesize
	&T_1:=\int_0^{t^*} \sqrt{u}dt = \sqrt{\frac{\rho \gamma}{\beta }} \int_0^{t^*} \sqrt{\frac{1}{ s }} dt \\
	&=2\sqrt{\frac{\rho }{\beta (s_0+\nu i_0)(\rho-1)^2\gamma }}\Bigg(\at\left(\sqrt{\frac{s_0}{s_0+\nu i_0}}\right) -\\
	&- \at\left(\sqrt{1-\frac{\nu i(t^*)}{s_0+\nu i_0}}\right)    \Bigg)\\
	&=2\sqrt{\frac{\rho }{\beta (s_0+\nu i_0)(\rho-1)^2\gamma }}   \Bigg(     \at\left(\sqrt{\frac{s_0}{s_0+\nu i_0}}\right) - \\
	&-\at\left(\sqrt{\frac{ s(t^*)}{s_0+\nu i_0}}\right)    \Bigg) 
	\\
	%&=2\sqrt{\frac{\rho }{\beta \tilde{\nu}(\rho-1)^2\gamma }}  { \left(     \at\left(\sqrt{\frac{s_0}{\tilde{\nu}}}\right) - \at\left(\sqrt{\frac{ s(t^*)}{\tilde{\nu}}}\right)    \right) }
	%\\
	&=2\sqrt{\frac{\rho }{\beta \tilde{\nu}(\rho-1)^2\gamma }}  { \left(     \at\left(\sqrt{\frac{s_0}{\tilde{\nu}}}\right) - \at\left(\sqrt{\frac{ \rho\gamma}{\beta\tilde{\nu}}}\right)    \right) }
\end{aligned}
\end{equation*}
where $\nu= \frac{\rho}{\rho-1}, \tilde{\nu}=(s_0+\nu i_0)$ and we used \eqref{eq:st*}.
Overall, the economic cost is
\begin{equation}
\int_0^\infty L dt =\int_0^{t^*} L dt = \int_0^{t^*} (1-\sqrt{u})dt = \left(t^*- T_1\right),
\label{eq:eco_cost_part1}
\end{equation}
where $t^*$ follows from \eqref{eq:explicit_part1}.
Plugging the economic cost \eqref{eq:eco_cost_part1}, the infected cost \eqref{eq:lin} and the infected square cost \eqref{eq:T_2} and \eqref{eq:i^2_A} in \eqref{eq:cost_parts}, we get \eqref{eq:t*}.
\subsection*{Proof of Proposition \ref{prpB}}
	$s(\tau_1)$ may be derived by \eqref{eq:constant} with $a=0$, $b=\tau_1$, i.e., it is solution of
	\begin{equation}
	\ln\Big(\frac{s(\tau_1)}{s_0}\Big)=\frac{\beta}{\gamma}[s(\tau_1)-s_0+\iota-i_0].
	\label{eq:s1}
	\end{equation}
	The control time $\Delta \tau:= \tau_2-\tau_1$ may be computed by using the fact that $i(t)=\iota$ for every $t \in [\tau_1,\tau_2]$, i.e.,
	\begin{equation}
		\tau_2-\tau_1=\frac{s(\tau_1)-s(\tau_2)}{\gamma \iota}=\frac{s(\tau_1)-\gamma/\beta}{\gamma \iota},
		\label{eq:delta_tau}
	\end{equation}
	From \cite[Theorem 1]{miclo2020optimal} and from our form of economic cost
\begin{equation}
	\label{eq:eco_w}
\begin{aligned}
	\int_0^\infty L dt &=\int_{\tau_1}^{\tau_2}\bigg[1-\sqrt{\frac{1}{1+(\tau_2-t)\beta\iota}}\bigg]dt\\
	&=\Big(\tau_2-\tau_1-\frac{2}{\beta \iota}(\sqrt{1+(\tau_2-\tau_1)\beta \iota}-1)\Big)\\
	&=\frac{1}{\iota}\bigg(\frac{s(\tau_1)}{\gamma}+\frac{1}{\beta}\Big(1-2\sqrt{\frac{s(\tau_1)\beta}{\gamma}}\Big)\bigg)\\
	&=\frac{1}{\iota}\bigg(\sqrt{\frac{s(\tau_1)}{\gamma}}-\sqrt{\frac{1}{\beta}}\bigg)^2.
\end{aligned}
\end{equation}
The infected cost is
	\begin{equation}
	\label{eq:linear_A}
	\int_0^\infty i(t)dt=\int_0^T\frac{\dot{r}(t)}{\gamma}dt=\frac{r(\infty)-r_0}{\gamma},
\end{equation}
where $r(\infty)$ follows by plugging $a=\tau_2$ in \eqref{eq:r_inf} with $i(t_2)=\iota$, $s(t_2)=\gamma/\beta$. For the infected square cost, for Phase I and III we use result in Lemma \ref{prp:nc_square}, while Phase II is simply $(\tau_2-\tau_1)\iota^2$. Hence, 
\begin{equation}
	\label{eq:i^2_W1}
	\begin{aligned}
		\int_0^{\infty} i^2 dt&= \gamma\frac{[\ln s_0]^2-[\ln s(\tau_1)]^2}{2\beta^2}-\frac{s_0-s(\tau_1)}{\beta} +\\ &+\big(s_0+i_0-\frac{\gamma}{\beta}\ln s_0\big)\frac{i_0+s_0-\iota-s(\tau_1)}{\gamma}+\\
		&+(\tau_2-\tau_1)\iota^2+\gamma\frac{[\ln \frac{\gamma}{\beta}]^2-[\ln s(\infty)]^2}{2\beta^2}-\frac{\frac{\gamma}{\beta}-s(\infty)}{\beta} +\\ &+\bigg(\frac{\gamma}{\beta}+\iota-\frac{\gamma}{\beta}\ln \frac{\gamma}{\beta}\bigg)\frac{r(\infty)-1+\iota+\frac{\gamma}{\beta}}{\gamma},
	\end{aligned}
\end{equation}
where $s(\tau_1)$ comes from \eqref{eq:s1}, $\tau_2-\tau_1$ from \eqref{eq:delta_tau}, and $s(\infty)=1-r(\infty)$. Summing infected square cost \eqref{eq:i^2_W1}, infected cost \eqref{eq:linear_A} and economic cost \eqref{eq:eco_w} we obtain \eqref{eq:cost_B}. 
\subsection*{Proof of Proposition \ref{prp:lb}}
	Let $z:=\beta s i (1-L)^2$, so that the dynamics is
	\begin{equation*}
		\dot{s}=-z, \quad
		\dot{i}=z -\gamma i, \quad
		\dot{r}=\gamma i.
	\end{equation*}
	Hence,
	\begin{equation*}
		c(t)=\Big(1-\sqrt{\frac{z}{\beta si}}\Big)+ \kappa \gamma_1 i^2  + \kappa\gamma_0 i.
	\end{equation*}
	We now show that
	\begin{equation}
		\label{eq:eps}
		1-\sqrt{\frac{z}{\beta si}} \ge \epsilon^2 \left(1- \frac{z}{\beta si}\right)^2,
	\end{equation}
	with $\epsilon = \sqrt{27/32}$, thus leading to
	\begin{equation*}
		c(t) \ge \bigg(\underbrace{\epsilon \Big(1- \frac{z}{\beta si}\Big)}_{c_1}\bigg)^2 + (\underbrace{\sqrt{\kappa \gamma_1}i}_{c_2})^2  + \underbrace{\kappa\gamma_0 i}_{c_3}.
	\end{equation*}
	Indeed, by defining $y:=\sqrt{z/(\beta s i)} \in [0,1]$, the largest $\epsilon$ satisfying \eqref{eq:eps} is
	$$
	\epsilon = \sqrt{\min_{y} \frac{1-\sqrt{y}}{(1-y)^2}}=\sqrt{\frac{27}{32}}.
	$$
	Hence,
	\begin{align*}
		\int_0^\infty c(t) \ge \int_0^\infty c_1^2(t)+c_2^2(t)+c_3(t)
	\end{align*}
	By assumption, it exists a time $t^*$ such that $s(t^*)=\frac{\gamma}{\beta}$. Hence,
	we can split the cost in two parts as follows:
	\begin{equation*}
		\int_0^\infty c_1^2(t)+c_2^2(t) = \int_0^{t^*} c_1^2(t)+c_2^2(t) + \int_{t^*}^\infty c_1^2(t)+c_2^2(t).
	\end{equation*}
	We start from the first part, i.e., $[0,t^*]$, using the fact that
	\begin{align*}
		\int_0^{t^*} c_1^2(t)+c_2^2(t) dt
		\ge  \int_0^{t^*} {2 c_1(t)c_2(t)} dt.
	\end{align*}
	For the second part, i.e., $[t^*,\infty]$, $L(t)=0$ by assumption (however, even if $L(t)>0$, this would be a lower bound), i.e.,
	\begin{equation*}
		\int_{t^*}^T c_1^2(t)+c_2^2(t) dt\ge \int_{t^*}^Tc_2^2(t) dt.
	\end{equation*}
	For convenience, we do not split $c_3$ in two parts.
	Putting all together, the total cost is
	\begin{equation}
		\int_0^T c(t) dt \ge 2 {U_{12}^I}+U_{22}^{II}+U_{3},
		\label{eq:cost_lower}
	\end{equation}
	%where $U$ denotes the integral of $c(t)$, with indexes $I$ and $II$ in $U$ denoting that the cost is integrated in $[0,t^*]$ and $[t^*,T]$, respectively, and indexes $12, 22, 3$ referring to the integration of $c_1c_2$, $c_2^2$, and $c_3$, respectively. 
	where $U_{12}^I:=\int_0^{t^*}c_1(t)c_2(t)dt$, $U_{22}^{II}:=\int_{t^*}^{\infty} c_2^2(t)dt$, and $U_3:=\int_{0}^{\infty} c_3(t)dt$.
	We now lower bound all of these terms as function of $i(t^*)$, and then minimize over $i(t^*)$ (which is function of the control itself) to get a lower bound. To this end, we establish the following equivalences.
	Let $A=\int_{0}^{t^*} i(t) dt$, and $B=\int_{0}^{t^*} \frac{z(t)}{\beta s(t)}dt$.
	\begin{enumerate}
		\item From $\dot{i}= z-\gamma i$ we get
		\begin{align*}
			\int_0^{t^*} \dot{i}(t) dt &= \int_0^{t^*} [z(t)-\gamma i(t)] dt \\
			i(t^*)-i_0&= \int_0^{t^*} z(t)dt - \gamma  A.
		\end{align*}
		Hence
		\begin{align*}
			A&=\int_0^{t^*} i(t) dt = \frac{-i(t^*)+ i_0+\int_0^{t^*} z(t)dt}{\gamma} \\
			&=\frac{-i(t^*)+ i_0-\int_0^{t^*} \dot{s}(t)dt}{\gamma} \\
			&=\frac{-i(t^*)+ i_0-s(t^*)+s_0}{\gamma}. 
		\end{align*}
		\item  From $\dot{s}=-z$ we get
		\begin{align*}
			&B=\int_0^{t^*}\frac{z(t)}{\beta s(t)} dt =-\frac{1}{\beta} \int_0^{t^*}\frac{\dot{s}(t)}{ s(t)} dt  \\
			&-\frac{1}{\beta} \int_0^{t^*}\frac{\partial ln(s(t))}{\partial t} dt = -\frac{1}{\beta} (ln(s(t^*))- ln(s_0)).
		\end{align*}
		
		%\item From Cauchy Schwartz, for any positive quantity $g$,
		%\begin{align*}
		%	\int_0^T \sqrt{g(t)} dt &\le \sqrt{\int_0^T g(t) dt} \sqrt{\int_0^T 1 dt} = \sqrt{T\int_0^T g(t) dt}  
		%\end{align*}
	\end{enumerate}
	We can now lower bound all the terms appearing in the cost. We recall that $i_0$ and $s_0$ are given, and $s(t^*)=\gamma/\beta$.
	\begin{align*}
		U_{12}^I &=  \int_0^{t^*} {c_1(t)c_2(t)} dt \\
		&= \int_0^{t^*}  \epsilon \sqrt{\kappa\gamma_1}\left(1-\frac{z(t)}{\beta s(t)i(t)}\right)  i(t) dt \\
		&=\epsilon \sqrt{\kappa\gamma_1}\Big[ \underbrace{\int_0^{t^*}   i(t)dt}_A- \underbrace{\int_0^{t^*}\frac{z(t)}{\beta s(t)} dt}_B \Big].
	\end{align*}
	Thus, using 1 and 2,
	\begin{equation}
		\begin{aligned}
			U_{12}^I&= \epsilon \sqrt{\kappa\gamma_1} \left[ \frac{ -i({t^*})+ i_0-s({t^*})+s_0}{\gamma}\right. \\ & \qquad \left. +\frac{1}{\beta} (ln({s(t^*)})- ln(s_0))\right].
		\end{aligned}
		\label{eq:U_12^I}
	\end{equation}
	$U_{22}^{II}$ can be computed exactly under the assumption $L(t)=0$ for $t\ge t^*$. %yielding $\lim_{t \rightarrow \infty}i(t) = 0$. 
	Hence, by Lemma \ref{prp:nc_square}, with $a=t^*$ and $b=\infty$, $s(t^*)=\gamma/\beta$, $i(t^*)$ parametric, $s(\infty)=1-r(\infty)$ being the asymptotic equilibrium computed by \eqref{eq:r_inf}, we obtain
	\begin{equation}
		\begin{aligned}
			U_{22}^{II} &=
			\frac{\gamma}{\beta}\frac{[\ln{(\gamma/\beta)}]^2-[\ln(s(\infty))]^2}{2\beta}-\frac{\gamma/\beta-s(\infty)}{\beta} +\\ &+\big(\gamma/\beta+i(t^*)-\frac{\gamma}{\beta}\ln {\gamma/\beta}\big)\frac{i(t^*)+\gamma/\beta-s(\infty)}{\gamma}.
		\end{aligned}
		\label{eq:U2^II}
	\end{equation}
	Finally, for $U_3$:
	\begin{equation}
		U_3=\kappa\gamma_0\int_0^{\infty}i(t)dt= \frac{\kappa\gamma_0}{\gamma}(r(\infty)-r_0).
		\label{eq:U3}
	\end{equation}
	Plugging \eqref{eq:U_12^I}, \eqref{eq:U2^II} and \eqref{eq:U3} in \eqref{eq:cost_lower}, we obtain the lower bound of the cost as a function of $i(t^*)$. To obtain a lower bound on the cost we should then minimize over $i(t^*)$. The statement follows from noticing that $i(t^*)$ is upper bounded by the peak of the infection in case of no control, i.e., $i_{max} = i_0+s_0-\frac{\gamma}{\beta}\big(1-\ln\frac{\gamma}{\beta s_0}\big)$ \cite{hethcote2000mathematics}.

\end{document}